\documentclass[11pt]{article}

\usepackage{amstext}
\usepackage{amssymb}
\usepackage{amsmath}
\usepackage{pb-diagram} \usepackage{amsthm}

\usepackage{rotating}

\usepackage{amsfonts}
\usepackage{graphicx}
\usepackage[skip=0pt]{caption}

\theoremstyle{plain}
 \newtheorem{thm}{Theorem}
 \newtheorem{lem}{Lemma}
 
 \newtheorem{prop}{Proposition}
 
\newtheorem{defn}{Definition}
 
\newtheorem{rem}{Remark}

\newcommand{\Y}{\includegraphics[width=1.3em]{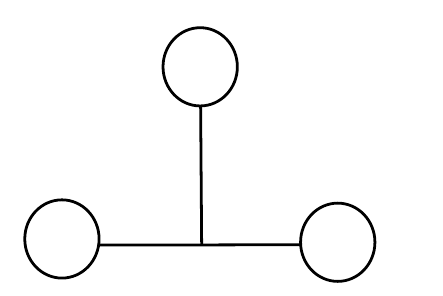}}

\begin{document}

\title{Classification and Models of Simply-connected Trivalent $2$-dimensional Stratifolds}

\author{J. C. G\'{o}mez-Larra\~{n}aga\thanks{Centro de
Investigaci\'{o}n en Matem\'{a}ticas, A.P. 402, Guanajuato 36000, Gto. M\'{e}xico. jcarlos@cimat.mx} \and F.
Gonz\'alez-Acu\~na\thanks{Instituto de Matem\'aticas, UNAM, 62210 Cuernavaca, Morelos,
M\'{e}xico and Centro de
Investigaci\'{o}n en Matem\'{a}ticas, A.P. 402, Guanajuato 36000,
Gto. M\'{e}xico. fico@math.unam.mx} \and Wolfgang
Heil\thanks{Department of Mathematics, Florida State University,
Tallahasee, FL 32306, USA. heil@math.fsu.edu}}
\date{}

\maketitle

\begin{abstract} Trivalent $2$-stratifolds are a generalization of $2$-manifolds in that there are disjoint simple closed curves where three sheets meet. We obtain a classification of $1$-connected $2$-stratifolds in terms of their associated labeled graphs and develop operations that will construct from a single vertex all graphs that represent $1$-connected $2$-stratifolds.\end{abstract}



\section{Introduction}   

In Topological Data Analysis  one studies high dimensional data sets by extracting shapes. Many of these shapes are $2$-dimensional simplicial complexes where it is computationally possible to calculate topological invariants such as the fundamental group or homology groups (see for example \cite{CL}). $2$-complexes that are amenable for more detailed analysis are foams and $2$-stratifolds. Foams include special spines of 3-dimensional manifolds \cite{M}, \cite{P}.  Khovanov \cite{Ko} used trivalent foams to construct a bigraded homology theory whose Euler characteristic is a quantum $sl(3)$ link invariant and Carter \cite{SC} presented an analogue of the Reidemeister-type moves for knotted foams in $4$-space. 

A closed $2$-{\it stratifold} is a $2$-dimensional cell complex $X$ that contains a collection of finitely many simple closed curves, the components of the $1$-skeleton $X^{(1)}$ of $X$, such that $X-X^{(1)}$ is a $2$-manifold and a neighborhood of each component $C$ of $X^{(1)}$ consists of $n\geq 3$ sheets. 

It is not not known which $3$-manifolds have spines that are $2$-stratifolds. There are significant differences: for any given  $2$-stratifold there are infinitely many non-homeomorphic $2$-stratifolds with the same fundamental group. $3$-manifold groups are residually finite, but every Baumslag-Solitar group (some of which are Hopfian, others are non-Hopfian) can be realized as the fundamental group of a simple $2$-stratifold. 

However it can be shown that fundamental groups of $2$-stratifolds have solvable word problem (work in progress). 
 
$2$-stratifolds arise in the study of categorical invariants of $3$-manifolds. For example if $\mathcal{G}$ is a non-empty family of groups that is closed under subgroups, one would like to determine which (closed) $3$-manifolds have $\mathcal{G}$-category equal to $3$. In \cite{GGL} it is shown that such manifolds have a decomposition into three compact $3$-submanifolds $H_1 ,H_2 ,H_3$ , where the intersection of $H_i \cap H_j$ (for $i\neq j$) is a compact $2$-manifold, and each $H_i$ is $\mathcal{G}$-contractible (i.e. the image of the fundamental group of each connected component of $H_i$ in the fundamental group of the $3$-manifold is in the family $\mathcal{G}$). The nerve of this decomposition, which is the union of all the intersections $H_i \cap H_j$ ($i\neq j$), is a closed $2$-stratifold and determines whether the $\mathcal{G}$-category of the 3-manifold is $2$ or $3$.

A $2$-stratifold is essentially determined by its associated bipartite labelled graph (defined in section 2) and a presentation for its fundamental group can be read off from the labelled graph. Thus the question arises when a labelled graph determines a simply connected $2$-stratifold. In \cite{GGH} it is shown that a necessary condition is that the underlying graph must be a tree; if the graph is linear then sufficient and necessary conditions on the labelling are given, and if the graph is trivalent (definition in section 2), an algorithm on the labelled graph was developed for determining whether the graph determines a simply connected $2$-stratifold. In \cite{GGH1} an algorithm is given that decides whether a given labelled graph (not necessarily trivalent) determines a $2$-stratifold that is homotopy equivalent to $S^2$. 

The main result of the present paper (in section 3) is a classification of all trivalent labelled graphs that represent simply connected $2$-stratifolds. Then in  sections 4 and 5 we develop three operations on labelled graphs that will construct from a single vertex all trivalent graphs that represent simply connected $2$-stratifolds.

\section{Properties of the graph of a $2$-stratifold.}

We first review the basic definitions and some results given in \cite{GGH} and \cite{GGH1}. A  $2$-{\it stratifold} is a compact, Hausdorff space $X$ that contains a closed (possibly disconnected) $1$-manifold $X^{(1)}$ as a closed subspace with the following property: Each  point $x\in X^{(1)}$  has a neighborhood homeomorphic to $\mathbb{R}{\times}CL$, where $CL$ is the open cone on $L$ for some (finite) set $L$ of cardinality $>2$  and $X - X^{(1)}$ is a (possibly disconnected) $2$-manifold.\\

A component $C\approx S^1$ of $X^{(1)}$ has a regular neighborhood $N(C)= N_{\pi}(C)$ that is homeomorphic to $(Y {\times}[0,1]) /(y,1)\sim (h(y),0)$, where $Y$ is the closed cone on the discrete space $\{1,2,...,d\}$ and $h:Y\to Y$ is a homeomorphism whose restriction to $\{1,2,...,d\}$ is the permutation $\pi:\{1,2,...,d\}\to  \{1,2,...,d\}$. The space $N_{\pi}(C)$ depends only on the conjugacy class of $\pi \in S_d$ and therefore is determined by a partition of $d$. A component of $\partial N_{\pi}(C)$ corresponds then to a summand of the partition determined by $\pi$. Here the neighborhoods $N(C)$ are chosen sufficiently small so that for disjoint components $C$ and $C'$ of $X^{(1)}$, $N(C)$ is disjoint from $N(C' )$. The components of $\overline{N(C)-C}$ are called the {\it sheets} of $N(C)$.\\

For a given $2$- stratifold $(X,X^{(1)} )$ there is an associated bipartite graph $G=G(X,X^{(1)} )$ embedded in $X$ as follows:\\

In each component $C_j$ of $X^{(1)}$ choose a black vertex $b_j$. In the interior of each component $W_i$ of $M=\overline{X-\cup_j N(C_j)}$ choose a white vertex $w_i$. In each component $S_{ij}$ of $W_i \cap N(C_j )$ choose a point $y_{ij}$, an arc $\alpha_{ij} $ in $W_i$ from $w_i$ to $y_{ij}$ and an arc $\beta_{ij}$ from $y_{ij}$ to $b_j$ in the sheet of $N(C_j )$ containing $y_{ij}$. An edge $e_{ij}$ between $w_i$ and $b_j$ consists of the arc $\alpha_{ij} *\beta_{ij}$. For a fixed $i$, the arcs $\alpha_{ij}$ are chosen to meet only at $w_i$.\\

We label the graph $G$ by assigning to a white vertex $W$ its genus $g$ of $W$ and by labelling an edge $S$ by $k$, where $k$ is the summand of the partition $\pi$ corresponding to the component $S$  of  $\partial N_{\pi}(C)$ where $S\subset \partial N_{\pi}(C)$. (Here we use Neumann's \cite{N} convention of assinging negative genus $g$ to nonorientable surfaces). Note that the partition $\pi$ of a black vertex is determined by the labels of the adjacent edges. If $G$ is a tree, then the labeled graph determines $X$ uniquely. \\

\noindent {\bf Notation}.  If $G$ is a bipartite labelled graph corresponding to the $2$-stratifold $X$ we let $X_G =X$ and $G_X =G$.\\

The following was shown in  \cite{GGH}.

\begin{prop}\label{retraction} There is a retraction $r:X\to G_X$. 
\end{prop}

 \begin{prop}\label{simplyconnected} If $X$ is simply connected, then $G_X$ is a tree, all white vertices of $G_X$ have genus $0$, and all terminal vertices are white.
\end{prop}

The proof uses Proposition \ref{retraction} and the following pruning construction:\\

\noindent {\bf Pruning at a subgraph}. For a subgraph $\Gamma$ of $G$ and the retraction $r:X_G \to G$, let $Y=r^{-1}(\Gamma )$. This is almost a $2$-stratifold, except that $Y$ has possibly boundary curves corresponding to edges of $st(\Gamma )-\Gamma$, where $st(\Gamma )$ is the star of $\Gamma$ in $G$. Let $\hat{Y}$ be the quotient of $X$ obtained by collapsing the closure of each component of $X_G  - Y$ to a point i.e. $\hat{Y}$ is obtained from $Y$ by attaching disks to its boundary curves. Then $\hat{Y}$ is a $2$-stratifold, $\hat{Y}=X_{\hat{\Gamma}}$, whose graph ${\hat \Gamma}$ is the union of $\Gamma$ and the labeled edges (with their vertices) of $st(\Gamma )-\Gamma$ which are adjacent to a black vertex of $\Gamma$. 

\begin{figure}[ht]
\begin{center}
\includegraphics[width=4in]{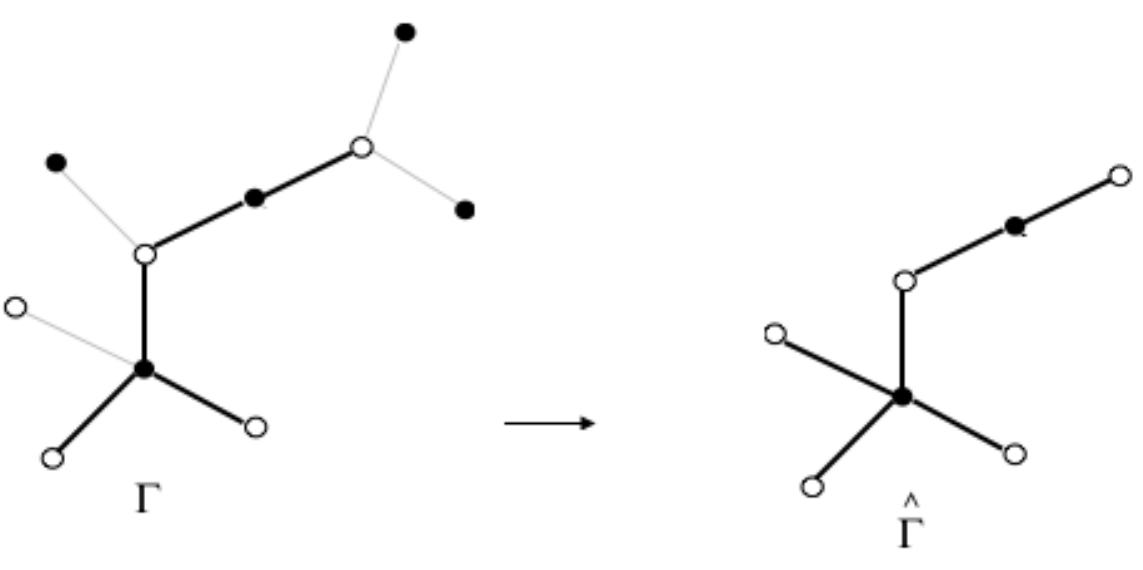} 
\end{center}
\caption{$G,\,\Gamma$ and ${\hat{\Gamma}}$}
\end{figure}

\begin{defn} $P$ is a pruned subgraph of $G$ if $P=\hat{\Gamma}$, for some subgraph $\Gamma$ of $G$.
\end{defn}

\begin{rem}\label{pruning} For a pruned subgraph $P$ of $G$, the quotient map $X_G \to X_P$ induces surjections $\pi (X_G )\to \pi (X_P)$ and $H_1 (X_G ;\mathbb{Z}_2 )\to H_1 (X_P ;\mathbb{Z}_2 )$. Therefore, if $X_G$ is simply connected, $X_P$ is simply connected. 
\end{rem}

In this paper we consider trivalent graphs. A $2$-stratifold $X$ and its  labeled bicolored graph $G_X$ are defined to be {\it trivalent}, if each black vertex $b$ is incident to either three edges each with label $1$ or to two edges, one with label $1$, the other with label $2$, or $b$ is a terminal vertex with adjacent edge of label $3$. This means that a neighborhood of a point of a component $C$ of the $1$-skeleton $X^{(1)}$ has $3$ sheets, so the permutation $\pi:\{1,2,3\}\to  \{1,2,3\}$ of the regular neighborhood $N(C)= N_{\pi}(C)$ has partition $1+1+1$ or $1+2$ or $3$.\\

One of the main results in \cite{GGH} (Theorem 5) is that a trivalent 2-stratifold is $1$-connected if and only if $H_1(X;\mathbb{Z}_2) = H_1(X;\mathbb{Z}_3) = 0$. The second condition is only needed to insure that $G_X$ has no terminal black vertices, therefore:

 \begin{thm}\label{pi=H} Let $X$ be a trivalent 2-stratifold such that all terminal vertices are white. Then $X$ is $1$-connected if and only if $H_1(X;\mathbb{Z}_2) = 0$.
\end{thm}

\section{Classification of $1$-connected trivalent $2$-stratifolds}

The building blocks for constructing labeled trivalent graphs for 1-connected $2$-stratifolds are called $b12$-graphs and $b111$-graphs:

\begin{defn} (1) The $b111$-tree is the bipartite tree consisting of one black vertex incident to three edges each of label $1$ and three terminal white vertices each of genus $0$.\\
(2) The $b12$-tree is the bipartite tree consisting of one black vertex incident to two edges one of label $1$, the other of label $2$, and two terminal white vertices each of genus $0$.
 \end{defn}
 \begin{figure}[ht]
\begin{center}
\includegraphics[width=2.5in]{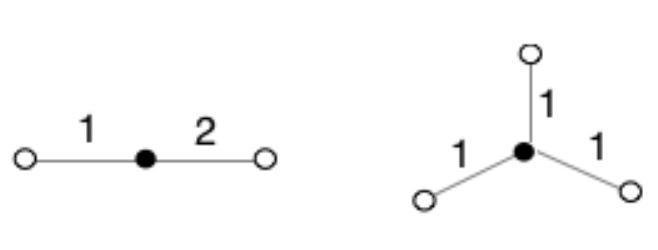}
\end{center}
\caption{\,$b12$-tree and $b111$-tree}
\end{figure}
First we consider special trivalent trees that do not contain $b111$-subtrees, which we call {\it barycentrically subdivided rooted $(2,1)$-labeled trees} or short {\it $(2,1)$-collapsible trees}:\\

A {\it barycentrically subdivided rooted $(2,1)$-labeled tree} is constructed from a rooted tree $T$ with root $r$ (a vertex of $T$) as follows: color with white and label $0$ the vertices of $T$,  take the barycentric subdivision $sd (T)$ of $T$ , color with black the new vertices (the barycenters of the edges of $T$) and finally label an edge $e$ of $sd (T)$ with $2$ (resp. $1$) if the distance from $e$ to the root $r$ is even (resp. odd). This labeled $sd (T)$ is the {\it $(2,1)$-collapsible tree} determined by $(T,r)$. We allow a one-vertex tree (with white vertex) as a $(2,1)$- collapsible tree.\\

A typical example is shown in Figure 3, where regions enclosed by the dashed curves are $(2,1)$-collapsible trees, bold white vertices are roots. 
 
\begin{lem}\label{12-components} Let $X_G$ be a trivalent $1$-connected $2$-stratifold. Then after removing the open stars of all black vertices of degree $3$, the components $C_1,\dots ,C_n$ are barycentrically subdivided rooted $(2,1)$-labeled trees. Furthermore, for each black vertex $b$ of degree $3$, at least one of its (three white) neighbors is the root of some $C_i$.
\end{lem}

\begin{figure}[ht]
\begin{center}
\includegraphics[width=4.5in]{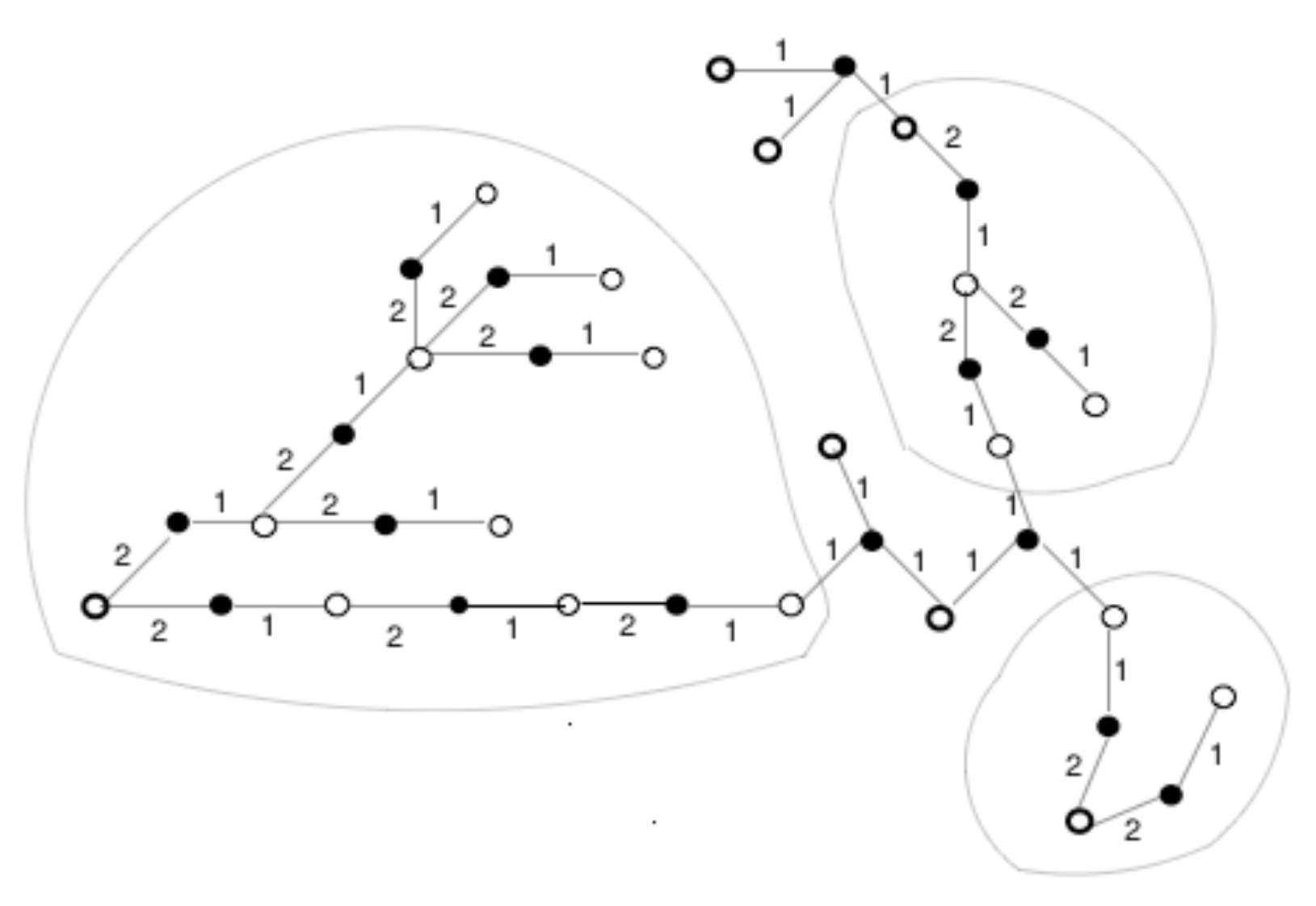}
\caption{}
\end{center}
\end{figure}

\begin{proof} Note that each $C_i$ is a pruned subgraph of $G$ and so $\pi(X_{C_i})=1$. Therefore for the first part of the proposition it suffices to show that if $C$ has no black vertices of degree $3$ and $X_C$ is simply connected, then $C$ is a barycentrically subdivided rooted $(2,1)$-labeled tree. 

By Proposition \ref{simplyconnected}, $C$ is a tree with all white vertices of genus $0$ and all terminal vertices white and, if $C$ is not a vertex,  contains a $b12$-subgraph $L$ with terminal edge of label $1$. Let $w$ be the white vertex of $L$ which is not a terminal vertex of $C$. Let $G_0 = (C-L)\cup \{w\}$, let $G_i' $ ($i=1,\dots ,m$) be the components of $G_0 -\{w\}$ and $G_i =G_i' \cup \{w\}$. Then $X_{G_i}$ is simply connected for $i=0,\dots ,m$. 
By induction on the number of vertices, each $G_i$ is a barycentrically subdivided rooted $(2,1)$-labeled tree with a root $r_i$. 

If $r_i =w$ for each $i=1,\dots ,m$ then $C$ is a barycentrically subdivided rooted $(2,1)$-labeled tree with root $w$. 

If $r_i \neq w$ the label on the edge $e_i$ of $G_i$ incident to $w$ is $1$. It follows that there is at most one $r_i$ not equal to $w$. For otherwise, if $r_i \neq w \neq r_j $ the unioin of the edges and vertices of the simple path in $G_i \cup G_j$ from $r_i$ to $r_j$ is a pruned linear subgraph $\Gamma=G(2,1,\dots , 1,1,2,1,\dots ,1,2)$ of $C$
for which $X_{\Gamma}$ is not simply connected by Theorem 3 of \cite{GGH}. This is a contradiction since $\pi_1 (X_{\Gamma})$ is a quotient of $\pi_1 (X_C )$. It follows that if $r_i \neq w$ then $C$ is a barycentrically subdivided rooted $(2,1)$-labeled tree with root $r_i$. 

For the second part of the proposition, suppose that $G$ contains a $b111$-subgraph with black vertex $b$ such that none of its white vertices $w_1 ,w_2 ,w_3$ is a root of any of the $C_j$ and let $C_i$ be the pruned subgraph of $G$ containing $w_i$ ($i=1,2,3$) and with roots $r_i \neq w_i$. Then the pruned subgraph $\Gamma$ of $G$ which is the union of the edges and vertices of the three simple paths in $G$ from the roots $r_i$ to $b$ ($i=1,2,3$) has all terminal edges of label $2$. By Lemma 4 of \cite{GGH}, $\pi_1 (X_{\Gamma})\neq 1$, a contradiction.
\end{proof}

The figure below shows that the converse of Lemma \ref{12-components} is false. There are two $b111$-vertices, all labels are $1$ except as indicated, and $H_1 (X_G ;\mathbb{Z}_2 )\cong \mathbb{Z}_2{\times}\mathbb{Z}_2$. Deleting the two $b12$-trees at the white center vertex (but not the center vertex) yields a {\it horned tree}.

\begin{figure}[ht]
\begin{center}
\includegraphics[width=4in]{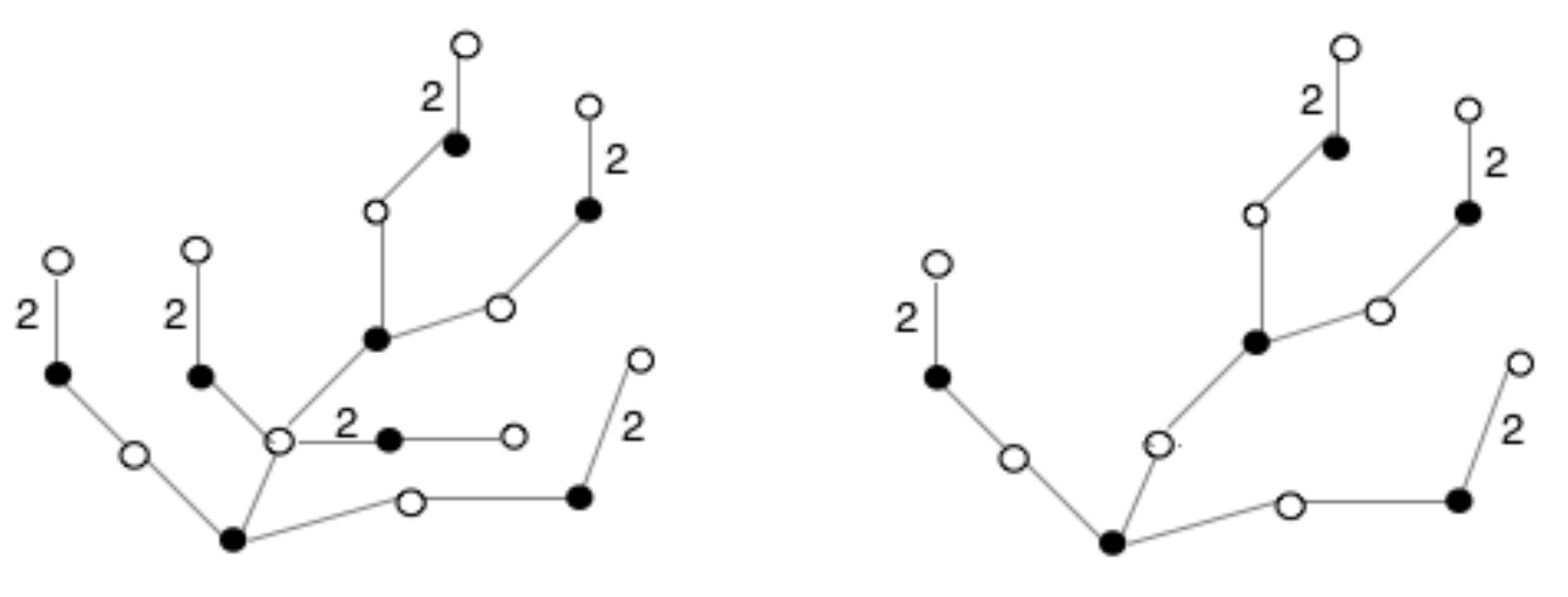}
\vspace{.1in}
\caption{\hspace{1in	}horned tree}
\end{center}
\end{figure}

\begin{defn} A horned tree $H_T$  is a finite connected bipartite labeled tree such that\\
(1) every black vertex $b$ whose distance to a terminal white vertex is $1$ has degree $2$; otherwise $b$ has degree 3;\\
(2) every nonterminal white vertex has degree $2$;\\
(3) every terminal edge has label $2$, every nonterminal edge has label $1$.\\
(4) there is at least one vertex of degree $3$.
\end{defn}
 A horned tree $H_T$ may be constructed from a tree $T$ all of whose nonterminal vertices have degree $3$, with at least one such vertex, as follows:
 
Color a vertex of $T$ white (resp. black) if it has degree $1$ (resp. $3$). Trisect the terminal edges of $T$ and bisect  the nonterminal edges, obtaining the graph $H_T$.
Color the additional vertices $v$ so that $H_T$ is bipartite, that is, $v$ is colored black if $v$ is a neighbor of a terminal vertex of $H_T$ and white otherwise. Then label the edges so that (3) holds.\\

The main property of $H_T$ is that $\pi_1 (X_{H_T}) = \mathbb{Z}_2$.\\

Finally we consider a ``reduced graph" of $G$ which encodes information on how the $(2,1)$-collapsible trees of Lemma \ref{12-components} are attached to the stars of the black vertices of degree $3$.\\

Denote by $B$ the union of all the black vertices of degree $3$ of $G$, let $St(B)$ be the (closed) star of $B$ in $G$ and $st(B)$ be the the open star of $B$.  Note that $G-st(B)$ consists of the components $C_1,\dots ,C_n$ as in Lemma \ref{12-components}.

\begin{defn} Let $G$ be a bipartite labeled tree such the the components of $G -st(B)$ are barycentrically subdivided rooted $(2,1)$-labeled trees. The reduced subgraph $R(G)$ of $G$ is the graph obtained from $St(B)$ by attaching to each white vertex $w$ of $St(B)$ that is not a root, a $b12$-graph such that the terminal edge has label $2$.
\end{defn}

As an example the reduced graph $R(G)$ for the graph in Figure 3 is shown below.
\begin{figure}[ht]
\begin{center}
\includegraphics[width=2in]{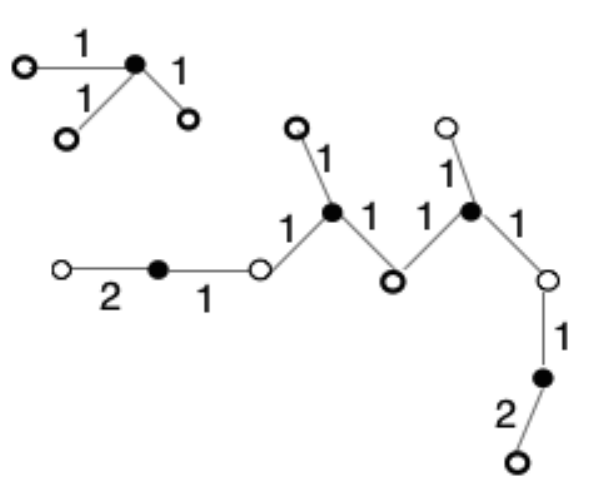}
\caption{}
\end{center}
\end{figure}

\begin{lem} \label{H(RG)} $H_1 (X_G ,\mathbb{Z}_2 )\approx H_1 (X_{R(G)} ,\mathbb{Z}_2 )$.
\end{lem}

\begin{proof} $H_1 (X_G ,\mathbb{Z}_2 )$ is generated by the black vertices. Let $w$ be a white vertex of $G$ of degree $n\geq 2$  with incident edges $e_1, \dots, e_n$. Suppose the label on $e_i$ is $2$ for $i=1, \dots, k$  ($k\leq n$). Split $w$ into $k+1$ (disjoint) vertices $w_1, \dots, w_k, w'$ so that each $w_k$ has degree $1$ with adjacent edge $e_k$ and $w'$ has degree $n-k$ with adjacent edges $e_{k+1}, \dots, e_n$. This change of the graph does not change $H_1 (X_G ,\mathbb{Z}_2 )$.

Now let C be a $(2,1)$-collapsible component of $G-st(B)$. If $w$ is a terminal (white) vertex of $C$ that is also a terminal vertex of $G$, delete the $b12$-subgraph of $C$ that contains $w$, if there is one ($C$ might consist of a single vertex). Continue doing this operation until all terminal vertices of $C$ belong to $St(B)$. This does not change $\pi_1 (X_G )$. If $w$ is a non-terminal white vertex of (the new $C$) then, if $w$ is the root, all edges incident to $w$ have label $2$; if $w$ is not the root,  all edges but one incident to $w$ have label $2$. Do the above construction on each such $w$ to change $C$ to $C'$. Then each component of $C'$ that does not contain a terminal white vertex of $C$ is a linear graph of type $w_1 - b_1 - \dots - w_k - b_k$ with successive edge labels $2-1-\dots-2-1$ (or consists of a single white vertex). Deleting these homologically trivial components from $C'$ does not change $H_1 (X_G ,\mathbb{Z}_2 )$.

Doing this for all $(2,1)$-collapsible components $C$ of $G-st(B)$ results in $R(G)$.
 \end{proof}

We now state the Classification Theorem.

\begin{thm} Let $X_G$ be a trivalent connected 2-stratifold. The following are equivalent :\\
(1) $X_G$ is 1-connected\\
(2) $G_X$ is a tree with all white vertices of genus $0$ and all terminal vertices white such that the components of $G -st(B)$ are barycentrically subdivided rooted $(2,1)$-labeled trees and the reduced graph $R(G)$ contains no horned tree. 
\end{thm}

\begin{proof}  If $X_G$ is 1-connected then by Proposition \ref{simplyconnected} and Lemma \ref{12-components} 
the components of $G - st(B)$ are barycentrically subdivided rooted $(2,1)$-labeled trees.
If the reduced graph $R(G)$ contains a horned tree $H$ let $C$ be the component  of $R(G)$ containing $H$. Note that $H$ is a pruned subgraph of $C$ and since $\pi_1 (H) = \mathbb{Z}_2$, it follows from Remark \ref{pruning} that $H_1 (C; \mathbb{Z}_2)\neq 0$. Lemma \ref{H(RG)} then shows that $H_1 (X_G ; \mathbb{Z}_2) \cong H_1 (X_{R(G)}; \mathbb{Z}_2 )\neq 0$. 
Hence $R(G)$ does not contain a horned tree.\\
   
Conversely, suppose the components of $G - st(B)$ are barycentrically subdivided rooted $(2,1)$-labeled trees and $R(G)$ contains no horned trees. Let $C$ be a component of $R(G)$. \\

First we show by induction on $n:=$ number of black vertices of degree $3$ in $C$, that $H_1 (X_C ;\mathbb{Z}_2 ) = 0$. \\
If $n=1$, then $C$ is a $b111$-tree with at most two $b21$-trees attached to its terminal vertices, and so 
$H_1 (X_C ; \mathbb{Z}_2 ) = 0$.\\
Let $n>1$. We claim that at least one terminal label of $C$ is $1$.\\
If not, then $C$ satisfies conditions (1) and (3)  of the definition of horned tree.
We can find a sequence $C=C_0, C_1, \dots,C_m $, where $C_m$ is a horned tree and $C_{i+1}$ is obtained from $C_i$ by deleting all but two components of $C_i  -  \{w\}$ for some nonterminal white vertex $w$ of $C_i$, contradicting the assumption that $R(G)$ contains no horned trees. This proves the claim.\\
 Now let $C' = C -st(b)$ where $b - w$ is a terminal edge of $C$ with label $1$. Then 
$H_1 (C; \mathbb{Z}_2 ) = H_1 (X_{C'}; \mathbb{Z}_2)$ which by induction is $0$. 

 Therefore $H_1 (X_G,\mathbb{Z}_2 ) = H_1 (X_{R(G)}; \mathbb{Z}_2 ) = 0$ and it follows from Theorem \ref{pi=H} that $X_G$ is $1$-connected.\\
\end{proof}
\section{Constructing trivalent graphs with all edge labels $1$.}

On a labeled graph $\Gamma=\Gamma_X$ with all white vertices of genus $0$ consider the following operation $O1$ that changes $\Gamma=\Gamma_X$ to $\Gamma_1 =\Gamma_{X_1}$:

\begin{figure}[ht]
\begin{center}
\includegraphics[width=3in]{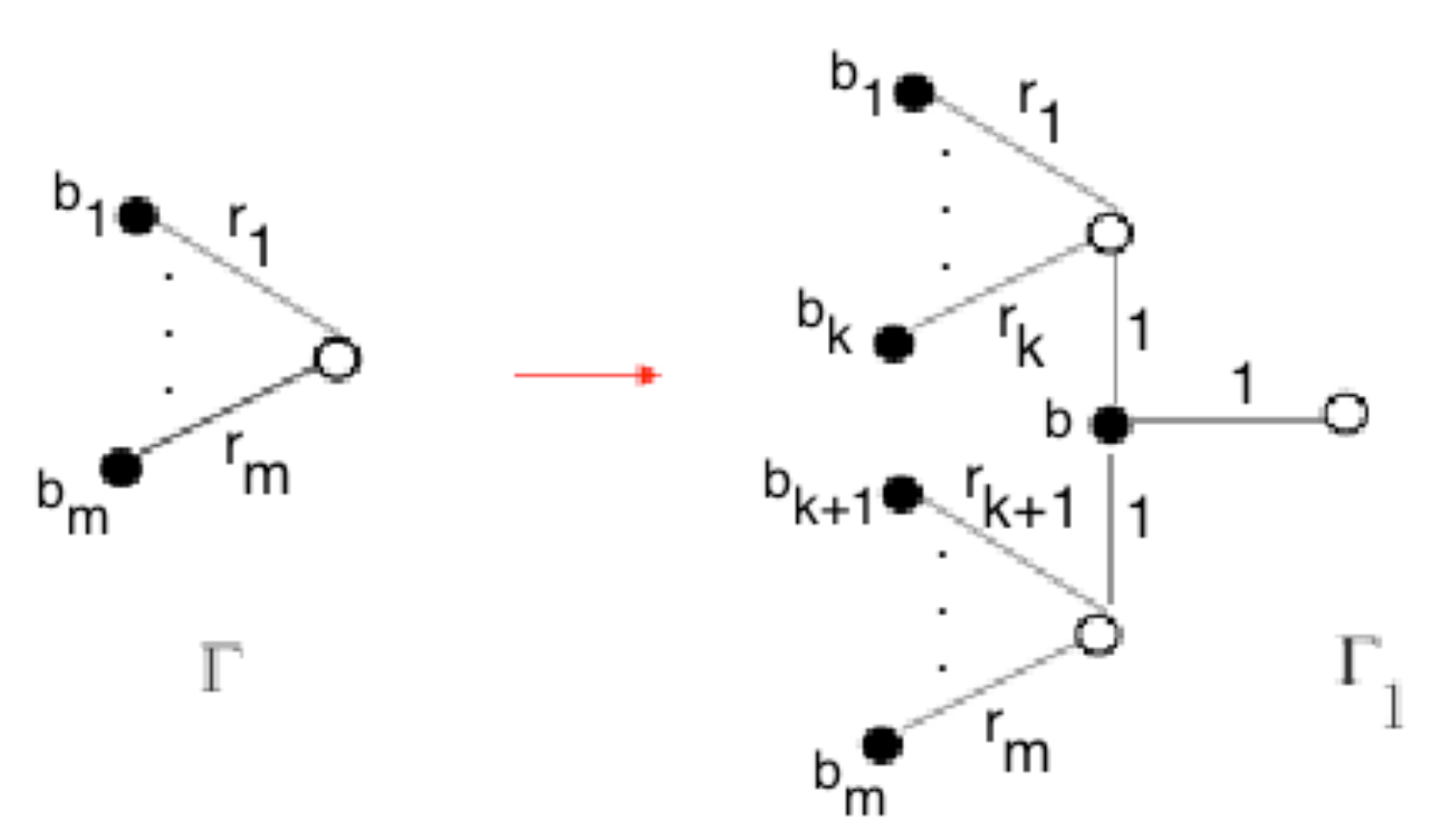} \qquad \qquad $m\geq k \geq 0$
\end{center}
\caption{\,Operation $O1$}
\end{figure}

If $b_i$ is the black vertex incident to the edge labeled $r_i$ then the corresponding relation $b_1^{r_1}\dots b_m^{r_m}=1$ in $\pi_1 (X)$ is changed to the relations $b_1^{r_1}\dots b_k^{r_k} b=1$, $b=1$,  $b_{k+1}^{r_{k+1}}\dots b_m^{r_m}=1$ in $\pi_1 (X_1 )$ and it follows that $\pi_1 (X_1 )$ is a quotient of $\pi_1 (X)$. In particular we note:
\begin{rem}\label{remop1} If $X$ is simpy connected, then operation $O1$ does not change the fundamental group.\end{rem}

\begin{lem}\label{op1} Let $X$ be a trivalent $2$-stratifold such that $\Gamma_X$ is a tree with all white vertices of label $0$, all terminal vertices white, and all edge labels $1$. Then $\Gamma_X$ can be reconstructed from any white vertex $w$ of $\Gamma_X$ by successively performing $O1$.
\end{lem}

\begin{proof} If $\Gamma=\Gamma_X$ consists of $w$ only, there is nothing to show. Clearly the Lemma is true if $\Gamma_X$ has only one black vertex.

Let $b$ be a black vertex incident to $w$. Deleting $b$ and its three incident edges from $\Gamma$, we obtain three subtrees $\Gamma'$, $\Gamma''$, $\Gamma'''$ that satisfy the conditions of the Lemma and with fewer black vertices than $\Gamma$. Denote by $w'$, $w''$, $w'''$ the three white vertices adjacent to $b$, where $w'=w\in \Gamma'$, $w''\in \Gamma''$ and $w'''\in \Gamma'''$. By induction on the number of black vertices, $\Gamma'$ is obtained from $w$ by repeated applications of operation $O1$.  Now apply $O1$ to $w\in \Gamma'$ to put back $b$ with its three edges and vertices $w''$ and $w'''$, then apply a sequence of $O1$'s to $w''$ and to $w'''$ to engulf $\Gamma''$ and $ \Gamma'''$.
\end{proof}

\begin{thm}\label{111model} Let $X$ be a trivalent $2$-stratifold such that each edge of $\Gamma_X$ has label $1$. Then the following are equivalent:\\
(1) $\pi_1 (X)=1$.\\
(2) $\Gamma_X$ is a tree with all white vertices of label $0$ and all terminal vertices white.\\
(3) $\Gamma_X$ can be constructed from the $b111$-graph by successively performing operation $O1$.
\end{thm}

\begin{proof} (1)$\implies$(2) by Corollary 1 of \cite{GGH}.\\
(2)$\implies$(3) by induction on the number of black vertice of $\Gamma_X$: If this number is $1$, then $\Gamma_X$ is a $b111$-graph, so suppose $\Gamma_X$ has at least two black vertices. By Corollary 1 of \cite{GGH}, $\Gamma_X$ contains a $b111$-subgraph $\Psi$ with a white vertex $w$ which is a terminal vertex of $\Gamma_X$. Let $b$ be the black vertex and $w'$, $w''$ the other white vertices of $\Psi$. Deleting the edges of $\Psi$ together with $b$ and $w$ splits $\Gamma_X$ into two subgraphs $\Gamma_{X'}$ and $\Gamma_{X''}$, each with fewer black vertices than $\Gamma_X$ and $w'\in \Gamma_{X'}$, $w''\in \Gamma_{X''}$. Now $X'$ and $X''$ are simply-connected and satisfy the conditions of Lemma \ref{op1}. By induction, $\Gamma_{X'}$ is obtained from $\Psi$ by successively performing operation $O1$. One further operation $O1$ (starting at $w'$) adds $\Psi$ to $\Gamma_{X'}$ and by Lemma \ref {op1} we can add $\Gamma_{X''}$ by performing successively operation $O1$, starting at $w''$.\\
(3)$\implies$(1): The $b111$-graph is simply connected and by Remark \ref{remop1} operation $O1$ does not change the fundamental group.
\end{proof}

\section{Constructing trivalent graphs with edge labels $1$ or $2$.}

For two disjoint labeled graphs $\Gamma_1=\Gamma_{X_1}$ and $\Gamma_2=\Gamma_{X_2}$ with all white vertices of genus $0$, operation $O1*$ described in Figure 7, creates a new graph $\Gamma=\Gamma_{X}$:\\

\begin{figure}[ht]{h}
\begin{center}
\includegraphics[width=4in]{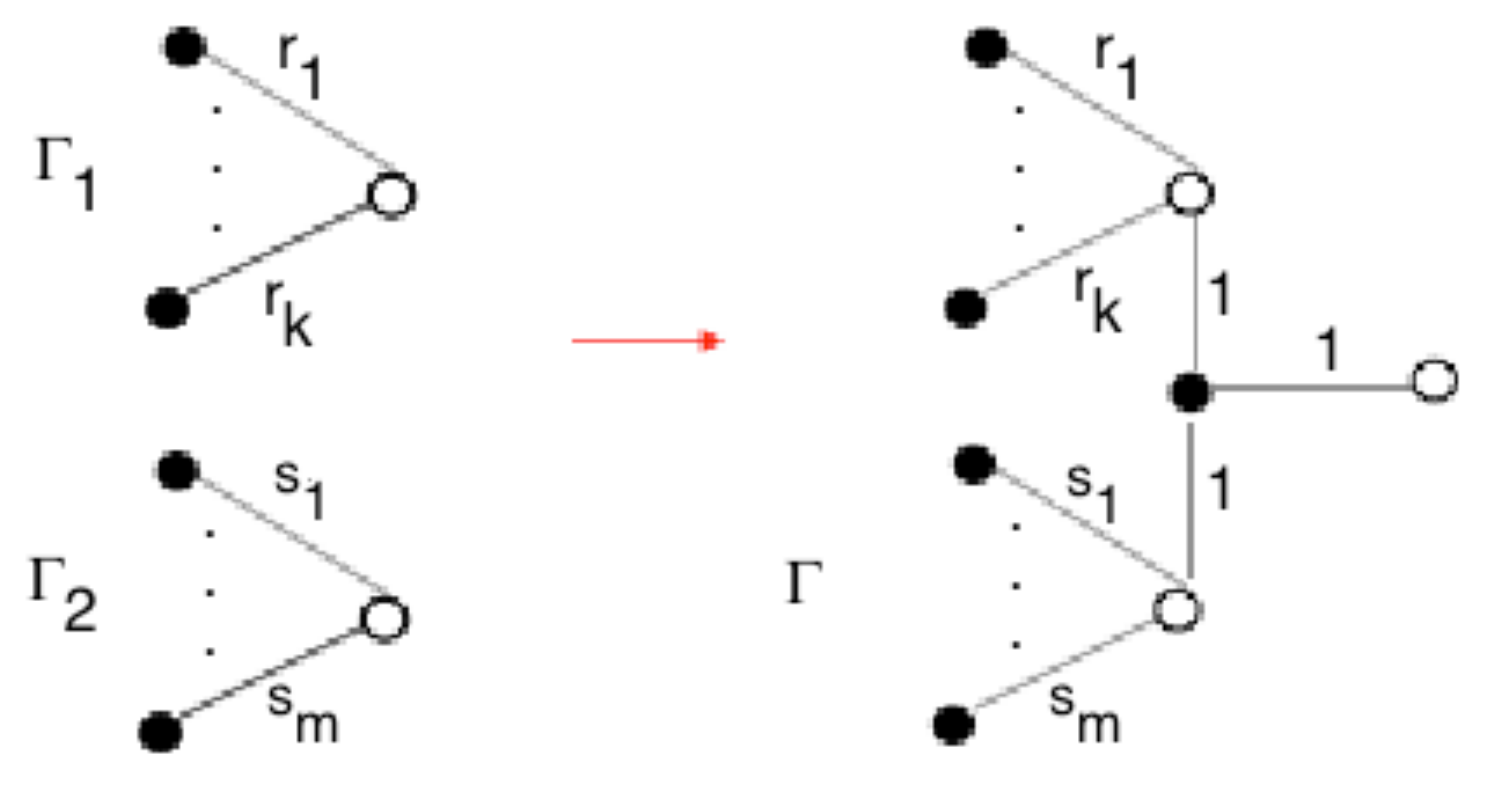}  \qquad \qquad $m \geq 0$
\end{center}
\caption{\,Operation $O1^*$}
\end{figure}

Note that $X$ is obtained from $X_1$ and $X_2$ by identifying a disk in $X_1$ with a disk in $X_2$, therefore :
\begin{rem}\label{remop1*} $\pi_1 (X) \cong \pi_1 (X_1 )*\pi_1 (X_2 )$.\end{rem}

Finally, on a labeled graph $\Gamma=\Gamma_X$ with all white vertices of genus $0$ consider operation  $O2$ described in Figure 4, that changes $\Gamma=\Gamma_X$ to $\Gamma_1 =\Gamma_{X_1}$:

\begin{figure}[ht]
\begin{center}
\includegraphics[width=3in]{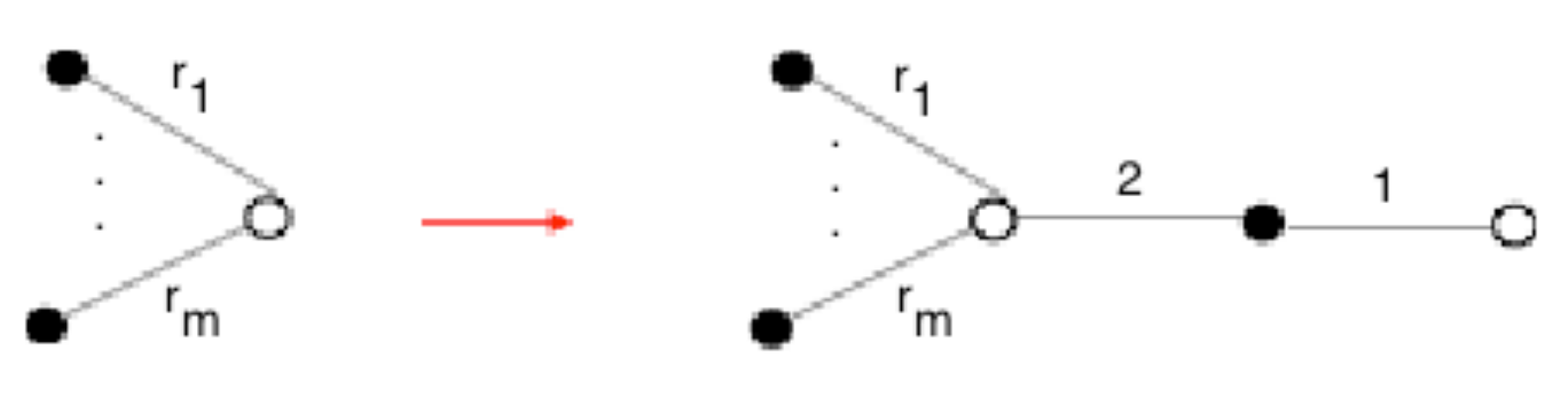}  \qquad \qquad $m \geq 0$
\end{center}
\caption{\,Operation $O2$}
\end{figure}

\begin{rem}\label{remop2} Operation $O2$ does not change the fundamental group.\end{rem}

We now describe the collection $\mathcal{G}$ of all trivalent graphs that can be obtained from a single white vertex by successively applying Operations $O1$ and $O2$.\\

For a collection $\mathcal{C}$ of bipartite labeled graphs denote by $\widehat{\mathcal{C}}$ the collection of all compact, connected bipartite labeled graphs obtained by starting with any $\Gamma_0 \in \mathcal{C}$ and successively performing Operations $O1$ or $O2$. We express this as\\

$\widehat{\mathcal{C}}= \{\emptyset, \Gamma_0 \stackrel{O^1}{\longrightarrow}\dots \stackrel{O^m}{\longrightarrow}
\Gamma \,\,\vert\,\, \Gamma_0 \in \mathcal{C},\,O^i =O1 \text{ or } O2\,;\,m\geq 0\,\}$\\

Let $\circ$ denote (the collection containing only) the graph consisting of one white vertex and let $\mathcal{G}_0 =\widehat{\circ}$\,.\\

For two connected bipartite labeled graphs $\Gamma$ and $\Gamma'$ denote by $\Gamma \Y \Gamma'$ a graph obtained by joining any white vertex of $\Gamma$ to any white vertex of $\Gamma'$ by operation $O1^*$. Note that that there are $v v'$ such $\Gamma \Y \Gamma'$, where $v$ (resp. $v'$) is the number of white vertices of $\Gamma$ (resp. $\Gamma'$). Let \\

$\mathcal{G}_0 \Y \mathcal{G}_0 =\{\,\Gamma \Y \Gamma'\,\vert\,\Gamma , \Gamma' \in \mathcal{G}_0\,\}$\\

In particular, $\mathcal{G}_0 \Y \emptyset =\mathcal{G}_0$ and $\emptyset \Y \emptyset =\emptyset$. Let \\

$\mathcal{G}_1  =\widehat{\mathcal{G}_0 \Y  \mathcal{G}_0}$, and inductively $\mathcal{G}_{n+1}  =\widehat{\mathcal{G}_n \Y  \mathcal{G}_n}$ \\

Then $\mathcal{G}_0 \subset \mathcal{G}_1 \subset \dots \subset \mathcal{G}_n  \subset \dots \,\,\subset \mathcal{G}:=\bigcup_{i=0}^{\infty} \mathcal{G}_i$

\begin{thm} Let $X$ be a trivalent $2$-stratifold. Then $X$ is simply connected if and only if $\Gamma_{X} \in \mathcal{G}$. 
\end{thm}

\begin{proof} If $\Gamma_X \in \mathcal{G}$ then $\pi_1 (X)=1$ by Remarks \ref{remop1*} and \ref{remop2}. 

Suppose $\pi_1 (X)=1$. If $\Gamma_X$ has no black vertices, or exactly one black vertex, then $\Gamma_X \in \mathcal{G}$. In any case $\Gamma_X$ is a tree with all white vertices of genus $0$ and all terminal edges white.  Furthermore by Lemma 4 of \cite{GGH},  $\Gamma_X$ contains a terminal vertex $w$ with incident edge $e$ of label $1$. Let $b$ be the black vertex incident to $e$. Then the star of $b$ in $\Gamma_X$ is a $b12$-graph or a $b111$-graph. In the first case star$(b)$ has two (open) edges  $e$,$e'$  where $e'$ has label $2$. The subgraph $\Gamma_X'$ obtained from $\Gamma_X$ by deleting $w\cup b \cup e \cup e'$ is simply connected (by Remark \ref{remop2}). By induction on the number of black vertices $\Gamma_{X}' \in \mathcal{G}$ and since $\Gamma_{X}$ is obtained from $\Gamma_{X}'$ by operation $O2$, it follows that $\Gamma_{X} \in \mathcal{G}$. 

In the second case, star($b$) has three edges  $e$,$e'$,$e''$, each with label $1$. The subgraphs $\Gamma_X'$ and $\Gamma_X''$ obtained from $\Gamma_X$ by deleting $w\cup b \cup e \cup e' \cup e''$ are simply connected by Remark \ref{remop1*}. By induction, $\Gamma_{X}' $ and $\Gamma_X''$ are in $\mathcal{G}$ and since $\Gamma_{X}$ is obtained from $\Gamma_{X}'$ and $\Gamma_X''$ by operation $O1^*$, it follows that $\Gamma_{X} \in \mathcal{G}$.
\end{proof}

{\bf Acknowledgments:} J. C. G\'{o}mez-Larra\~{n}aga would like to thank LAISLA and INRIA Saclay for financial support  and INRIA Saclay and IST Austria for their hospitality.


\begin{thebibliography}{99}



\bibitem {SC} J.S. Carter, Reidemeister/Roseman-type moves to embedded foams in 4-dimensional space. arXiv:1210.3608v1 [math.GT] 

\bibitem {GGL} J.C. G\'{o}mez-Larra\~{n}aga, F. Gonz\'alez-Acu\~na, Wolfgang Heil, Categorical group invariants of 3-manifolds, manuscripta math. 145, 433-448 (2014).

\bibitem {GGH} J.C. G\'{o}mez-Larra\~{n}aga, F. Gonz\'alez-Acu\~na, Wolfgang Heil, $2$-stratifolds, in ``A Mathematical Tribute to Jos\'e Mar\'ia Montesinos Amilibia", Universidad Complutense de Madrid, 395-405 (2016).

\bibitem {GGH1} J.C. G\'{o}mez-Larra\~{n}aga, F. Gonz\'alez-Acu\~na, Wolfgang Heil, $2$-dimensional stratifolds homotopy equivalent to $S^2$, Topology Appl. 209, 56-62 (2016).

\bibitem {Ko} M. Khovanov,  sl(3) link homology. Algebr. and Geom. Topol. 4, 1045-1081 (2004).

\bibitem {CL} P. Lum, G. Singh, J. Carlsson, A. Lehman, T. Ishkhanov, M. Vejdemo-Johansson, M. Alagappan, G. Carlsson, Extracting insights from the shape of complex data using topology. Nature Scientific Reports 3, 12-36 (2013).

\bibitem {M} S.V. Matveev,  Distributive groupoids in knot theory; (Russian) Mat. Sb. (N.S.) 119(161)
(1982), no. 1, 78-88, 160.

\bibitem {P} R. Piergallini, Standard moves for standard polyhedra and spines, Third National Conference
on Topology (Italian) (Trieste, 1986). Rend. Circ. Mat. Palermo (2) Suppl. No. 18 (1988),
391-414.

\bibitem {N}  W. Neumann, A calculus for plumbing applied to the topology of complex surface singularities and degenerating complex curves, Trans. Amer. Math. Soc. 268, 299-344 (1981).




\end{thebibliography}
\end{document}